\documentclass[12pt]{article}
\usepackage{latexsym, amssymb, amsthm, amsmath}
\usepackage{dsfont}
\usepackage{graphicx}
\usepackage{color}

\DeclareMathAlphabet\gothic{U}{euf}{m}{n}

\makeatletter
\def\eqnarray{\stepcounter{equation}\let\@currentlabel=\theequation
\global\@eqnswtrue
\tabskip\@centering\let\\=\@eqncr
$$\halign to \displaywidth\bgroup\hfil\global\@eqcnt\z@
  $\displaystyle\tabskip\z@{##}$&\global\@eqcnt\@ne
  \hfil$\displaystyle{{}##{}}$\hfil
  &\global\@eqcnt\tw@ $\displaystyle{##}$\hfil
  \tabskip\@centering&\llap{##}\tabskip\z@\cr}

\def\endeqnarray{\@@eqncr\egroup
      \global\advance\c@equation\m@ne$$\global\@ignoretrue}

\def\@yeqncr{\@ifnextchar [{\@xeqncr}{\@xeqncr[5pt]}}
\makeatother

\allowdisplaybreaks[3]

\begin{document}
\bibliographystyle{tom}

\newtheorem{theorem}{Theorem}[section]
\newtheorem{thm}[theorem]{Theorem}
\newtheorem{lemma}[theorem]{Lemma}
\newtheorem{corollary}[theorem]{Corollary}
\newtheorem{cor}[theorem]{Corollary}
\newtheorem{proposition}[theorem]{Proposition}
\newtheorem{prop}[theorem]{Proposition}
\newtheorem{crit}[theorem]{Criterion}
\newtheorem{alg}[theorem]{Algorithm}

\theoremstyle{definition}

\newtheorem{definition}[theorem]{Definition}
\newtheorem{assu}[theorem]{Assumption}
\newtheorem{conj}[theorem]{Conjecture}
\newtheorem{exmp}[theorem]{Example}
\newtheorem{exam}[theorem]{Example}
\newtheorem{prob}[theorem]{Problem}
\newtheorem{rem}[theorem]{Remark}
\newtheorem{remarkn}[theorem]{Remark}

\theoremstyle{remark}

\newtheorem*{ack}{Acknowledgement}  



\newcommand{\gota}{\gothic{a}}
\newcommand{\gotb}{\gothic{b}}
\newcommand{\gotc}{\gothic{c}}
\newcommand{\gote}{\gothic{e}}
\newcommand{\gotf}{\gothic{f}}
\newcommand{\gotg}{\gothic{g}}
\newcommand{\gothh}{\gothic{h}}
\newcommand{\gotk}{\gothic{k}}
\newcommand{\gotl}{\gothic{l}}
\newcommand{\gotm}{\gothic{m}}
\newcommand{\gotn}{\gothic{n}}
\newcommand{\gotp}{\gothic{p}}
\newcommand{\gotq}{\gothic{q}}
\newcommand{\gotr}{\gothic{r}}
\newcommand{\gots}{\gothic{s}}
\newcommand{\gott}{\gothic{t}}
\newcommand{\gotu}{\gothic{u}}
\newcommand{\gotv}{\gothic{v}}
\newcommand{\gotw}{\gothic{w}}
\newcommand{\gotz}{\gothic{z}}
\newcommand{\gotA}{\gothic{A}}
\newcommand{\gotB}{\gothic{B}}
\newcommand{\gotG}{\gothic{G}}
\newcommand{\gotL}{\gothic{L}}
\newcommand{\gotS}{\gothic{S}}
\newcommand{\gotT}{\gothic{T}}

\newcounter{teller}
\renewcommand{\theteller}{(\alph{teller})}
\newenvironment{tabel}{\begin{list}%
{\rm  (\alph{teller})\hfill}{\usecounter{teller} \leftmargin=1.1cm
\labelwidth=1.1cm \labelsep=0cm \parsep=0cm}
                      }{\end{list}}

\newcounter{tellerr}
\renewcommand{\thetellerr}{(\roman{tellerr})}
\newenvironment{tabeleq}{\begin{list}%
{\rm  (\roman{tellerr})\hfill}{\usecounter{tellerr} \leftmargin=1.1cm
\labelwidth=1.1cm \labelsep=0cm \parsep=0cm}
                         }{\end{list}}

\newcounter{tellerrr}
\renewcommand{\thetellerrr}{(\Roman{tellerrr})}
\newenvironment{tabelR}{\begin{list}%
{\rm  (\Roman{tellerrr})\hfill}{\usecounter{tellerrr} \leftmargin=1.1cm
\labelwidth=1.1cm \labelsep=0cm \parsep=0cm}
                         }{\end{list}}

\newcounter{proofstep}
\newcommand{\nextstep}{\refstepcounter{proofstep}\vertspace \par 
          \noindent{\bf Step \theproofstep} \hspace{5pt}}
\newcommand{\firststep}{\setcounter{proofstep}{0}\nextstep}

\newcommand{\Ni}{\mathds{N}}
\newcommand{\Qi}{\mathds{Q}}
\newcommand{\Ri}{\mathds{R}}
\newcommand{\Ci}{\mathds{C}}
\newcommand{\Ti}{\mathds{T}}
\newcommand{\Zi}{\mathds{Z}}
\newcommand{\Fi}{\mathds{F}}
\newcommand{\field}[1]{\mathbb{#1}}
\newcommand{\I}{\field{I}}
\newcommand{\N}{\field{N}}
\newcommand{\R}{\field{R}}
\newcommand{\C}{\field{C}}

\renewcommand{\proofname}{{\bf Proof}}

\newcommand{\vertspace}{\vskip10.0pt plus 4.0pt minus 6.0pt}

\newcommand{\ad}{{\mathop{\rm ad}}}
\newcommand{\Ad}{{\mathop{\rm Ad}}}
\newcommand{\clalg}{{\mathop{\overline{\rm alg}}}}
\newcommand{\Aut}{\mathop{\rm Aut}}
\newcommand{\arccot}{\mathop{\rm arccot}}
\newcommand{\capp}{{\mathop{\rm cap}}}
\newcommand{\rcapp}{{\mathop{\rm rcap}}}
\newcommand{\diam}{\mathop{\rm diam}}
\newcommand{\divv}{\mathop{\rm div}}
\newcommand{\codim}{\mathop{\rm codim}}
\newcommand{\RRe}{\mathop{\rm Re}}
\newcommand{\IIm}{\mathop{\rm Im}}
\newcommand{\Tr}{{\mathop{\rm Tr \,}}}
\newcommand{\Vol}{{\mathop{\rm Vol}}}
\newcommand{\card}{{\mathop{\rm card}}}
\newcommand{\rank}{\mathop{\rm rank}}
\newcommand{\supp}{\mathop{\rm supp}}
\newcommand{\sgn}{\mathop{\rm sgn}}
\newcommand{\essinf}{\mathop{\rm ess\,inf}}
\newcommand{\esssup}{\mathop{\rm ess\,sup}}
\newcommand{\osc}{{\mathop{\rm osc}}}
\newcommand{\Int}{\mathop{\rm Int}}
\newcommand{\lcm}{\mathop{\rm lcm}}
\newcommand{\loc}{{\rm loc}}
\newcommand{\HS}{{\rm HS}}
\newcommand{\Trn}{{\rm Tr}}

\newcommand{\tr}{\operatorname{tr}}     
\newcommand{\dom}{\operatorname{Dom}}   
\newcommand{\dist}{\operatorname{dist}}
\newcommand{\mes}{\operatorname{mes}}   

\newcommand{\BR}{\color{blue}}
\newcommand{\ER}{\color{black}}

\newcommand{\at}{@}

\newcommand{\spann}{\mathop{\rm span}}
\newcommand{\one}{\mathds{1}}

\hyphenation{groups}
\hyphenation{unitary}

\newcommand{\ca}{{\cal A}}
\newcommand{\cb}{{\cal B}}
\newcommand{\cc}{{\cal C}}
\newcommand{\cd}{{\cal D}}
\newcommand{\ce}{{\cal E}}
\newcommand{\cf}{{\cal F}}
\newcommand{\ch}{{\cal H}}
\newcommand{\chs}{{\cal HS}}
\newcommand{\ci}{{\cal I}}
\newcommand{\ck}{{\cal K}}
\newcommand{\cl}{{\cal L}}
\newcommand{\cm}{{\cal M}}
\newcommand{\cn}{{\cal N}}
\newcommand{\co}{{\cal O}}
\newcommand{\cp}{{\cal P}}
\newcommand{\cs}{{\cal S}}
\newcommand{\ct}{{\cal T}}
\newcommand{\cx}{{\cal X}}
\newcommand{\cy}{{\cal Y}}
\newcommand{\cz}{{\cal Z}}
\newcommand{\m}{{\mu}}

\thispagestyle{empty}

\vspace*{1cm}
\begin{center}
{\large\bf H\"older estimates for parabolic operators \\on domains with rough boundary}
\\[5mm]
\large K. Disser (Berlin), A.F.M. ter Elst (Auckland) and J. Rehberg (Berlin)
\end{center}

\vspace{5mm}

\begin{list}{}{\leftmargin=1.8cm \rightmargin=1.8cm \listparindent=10mm
   \parsep=0pt}
\item
\small
{\sc Abstract}.
In this paper we investigate linear parabolic, second-order boundary
value problems with mixed boundary conditions on rough domains. 
Assuming only 
boundedness/ellipticity on the coefficient function and very mild
conditions on the geometry of the domain -- including a very weak
compatibility condition between the Dirichlet boundary part and its
complement -- we prove H\"older continuity of the solution in space and 
time. 

\end{list}

\let\thefootnote\relax\footnotetext{
\begin{tabular}{@{}l}
\end{tabular}}


\section{Introduction} \label{Sholder1}
This
paper is concerned with parabolic initial-boundary value problems including mixed 
boundary conditions of the type
\begin{align}\label{e-prob}
u'(t,x)-\mathrm{div}(\mu(x)\nabla u)(t,x) & = f(t,x),& \text{in }(0,T)\times\Omega, \\
u(t,x) &= 0, & \text{on }(0,T)\times D, \nonumber\\ 
\mu(x) (\nabla u)(t,x) \cdot \nu(x) &= 0, & \text{on }(0,T)\times \Upsilon, \nonumber\\
u(0,x) &= u_0, &\text{in }\Omega,\nonumber
\end{align} 
where $D\subset \partial \Omega$ and $\Upsilon =\partial\Omega \setminus D$ are 
Dirichlet and Neumann boundary parts for the domain $\Omega\subset \R^d$ with 
outer normal vectors $\nu$ and $d \geq 2$. 
We show that both for all $f \in L^s((0,T);L^p(\Omega))$, with
$p\in (\frac{d}{2},\infty)$ and for all $f \in L^s((0,T);W^{-1,q}_D(\Omega))$, 
with $q \in (d,\infty)$, and $s$ sufficiently large,
the problem is well-posed and there exists a $\beta > 0$ such that the solution satisfies
$u \in C^\beta((0,T)\times \Omega)$, that is the solution
is H\"older continuous in space and time. 

H\"older continuity is one of the classical features in the theory for parabolic equations, 
where we refer to the initial work of Nash and Moser, \cite{Nash}, \cite{Mos},
\cite{Mos2} and to the extensive theory for parabolic initial-boundary
value problems developed in the monograph
\cite{LaSU}.
One of the main reasons for proving H\"older estimates is their usefulness in the 
investigation of nonlinear problems.
In \cite[p.~9]{LaSU} the domain $\Omega$ is assumed to be
`piecewise $C^1$ with nonzero interior angles'.
In the 
standard work of Lieberman \cite{Liebe3}, the domain is assumed to be Lipschitz. 
The main novelty of our results lies in reducing the assumptions on the parts 
$D$ and $\Upsilon$ of the boundary $\partial \Omega$ to include rough settings.
For the Dirichlet part $D$, we merely require the outer volume 
condition (see e.g.\
\cite[Chapter~II Theorem~B.4]{KinS}), which is classical for the elliptic pure Dirichlet problem.
In particular, the domain may be rough in that the inner volume condition is not required. 
The second achievement is that we can considerably weaken the conditions on the 
relative boundary of the Dirichlet and Neumann boundary part in that we replace the 
geometrical condition established in \cite{Groe} 
(compare also \cite{GKR}
\cite{HiebR}, \cite{Grie2}, \cite{Grie4}) by a measure theoretic one.
Roughly speaking, it states that in balls around points in the intersection 
$D \cap \overline{\Upsilon}$, the  Dirichlet boundary part is not rare 
(in a certain quantitative sense)
with respect to the boundary measure, see (\ref{e-masstheo}) below. 
This reflects the fact that H\"older continuity
for the elliptic Dirichlet problem also requires only a measure theoretic assumption \cite{Stam2}. 
Our framework is thus much broader than the classical one and allows for interesting new cases.
In particular, the Dirichlet boundary part need not be (part of) a continuous boundary
in the sense of \cite[Definition~1.2.1.1]{Gris} and the domain is not required to `lie on one 
side of the Dirichlet boundary part', see Figure \ref{fig}.

\begin{figure}
\centerline{\includegraphics[scale=0.25]{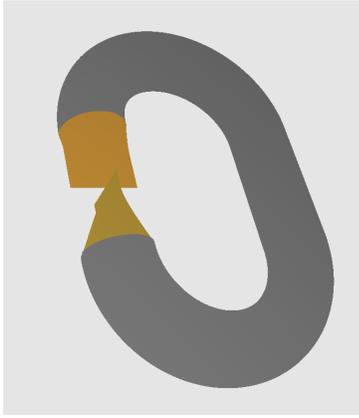}}
\caption{\label{fig} {The light coloured part of the boundary carries Dirichlet boundary conditions,
 the dark coloured part has Neumann boundary conditions.}}
\end{figure}

Under these more general assumptions on the geometry, we essentially reproduce the classical
parabolic H\"older theory in \cite{LaSU}, however, in our case, in the standard H\"older spaces, 
but with the coefficients independent of time. 
This rests on the fact that our prescribed Dirichlet data are identically zero, 
whereas in \cite{LaSU}, more general data are admitted. 

Our paper is an extension of the results provided in \cite{ERe2}, 
where the geometric setting was developed and where H\"older regularity for the 
elliptic problem as well as Gaussian H\"older kernel bounds on the semigroup were proved.
We show that space-time H\"older continuity for the parabolic problem with 
$f \in L^s((0,T);L^p(\Omega))$ follows essentially by employing maximal parabolic 
regularity and interpolation. 

The second main point of this paper is to study the case $f \in L^s((0,T);W^{-1,q}_D(\Omega))$ 
and provide a similar result.  
A motivation for including distributional right-hand-sides is given at the 
beginning of Section~\ref{s-parabolicLp} and we refer to Section~\ref{s-pr} for a 
precise definition of $W^{-1,q}_D(\Omega)$.
The method of proof transfers from the $L^p$-setting to the $W_D^{-1,q}$-setting 
due to an abstract relation of the fractional powers of the elliptic operator 
considered in $L^p$ and in $W_D^{-1,q}$, respectively, see Lemma~\ref{l-domApAq}.
We prove this property under slightly less general assumptions on $D$ and on the 
coefficient function $\mu$, cf.\ Assumptions~\ref{a-Ahlf} and \ref{assu-coeffi}.

The outline of the paper is as follows. 
In Section~\ref{s-pr} we provide basic definitions, the 
main assumptions and preliminary results. 
In Section~\ref{s-parabolicLp} we study problem \eqref{e-prob} in 
$L^s((0,T);L^p(\Omega))$ and prove that 
solutions are H\"older continuous in space and time. 
In Section~\ref{SHoelpara4} a similar result is proved for $f\in L^s((0,T);W^{-1,q}_D(\Omega))$.

\begin{ack}
	K. Disser was supported by the European 
	Research Council via ERC-2010-AdG no.\ 267802 \emph{(Analysis of
	Multiscale Systems Driven by Functionals)}.
	\end{ack}


\section{Preliminary results}\label{s-pr}

Fix $d \in \{ 2,3,\ldots \} $.
Let $\Omega \subset \R^d$ be a bounded domain and let $D$ be a 
closed subset of $\partial \Omega$.
We define
\[
C^\infty_D(\Omega)
:=\{w|_\Omega: w \in C_c^\infty(\R^d) \mbox{ and } 
   \supp w \cap D =\emptyset \}.
\]
Note that if $D =\partial \Omega$, then $C_{\partial \Omega}^\infty(\Omega) = C_c^\infty(\Omega)$.
For all $p \in [1,\infty)$, we denote the closure of $C^\infty_D(\Omega)$ in 
$W^{1,p}(\Omega)$ by $W^{1,p}_D(\Omega)$, where $W^{1,p}(\Omega)$ is the usual complex 
Sobolev space of order~$1$. 
If $p \in (1,\infty]$, then
the space $W^{-1,p}_D(\Omega)$ is the anti-dual of $W^{1,p'}_D(\Omega)$
in $L^p(\Omega)$, where 
$p'$ is the conjugate index for $p$, so $\frac{1}{p} + \frac{1}{p'} = 1$.
The domain $\Omega$ remains fixed throughout the paper, and hence we
omit $\Omega$ in the notation of all function spaces. 
For example, we write $L^p$ instead of $L^p(\Omega)$.

We always assume that the coefficient function $\mu \colon \Omega \to \R^{d \times d}$
is bounded, measurable and satisfies the ellipticity condition, that is,
there exists a $\underline \mu > 0$ such that 
\[
\RRe \sum_{i,j=1}^d \mu_{ij}(x) \, \xi_i \, \overline{\xi_j} 
\geq \underline \mu \, |\xi|^2
\]
for almost all $x \in \Omega$ and for all $\xi \in \Ci^d$, where $\mu_{ij}(x)$ denote the 
matrix coefficients of $\mu(x)$ in Euclidean coordinates.
We define the sesquilinear 
form $\gotl \colon W^{1,2}_D \times W^{1,2}_D \to \Ci$ by
\[
\gotl(u,v)
= \int_\Omega \sum_{i,j=1}^d \mu_{ij} \, (\partial_i u) \, \overline{(\partial_j v)}
{}.  \]
Then $\gotl$ is closed and sectorial.  
Next we define $\ca \colon W^{1,2}_D \to W^{-1,2}_D$ by 
\[
\langle \ca u,v\rangle = \gotl(u,v).
\]
If $q \in (2,\infty)$, then define the operator $\ca_q \colon \dom(\ca_q) \to W^{-1,q}_D$
by 
\[
\dom(\ca_q) = \{ u \in W^{1,2}_D : \ca u \in W^{-1,q}_D \} 
\]
and $\ca_q u = \ca u$ for all $u \in \dom(\ca_q)$.
We consider $\ca_q$ as an unbounded operator in $W^{-1,q}_D$.
Similarly, let $A$ be the $m$-sectorial operator associated with $\gotl$ in $L^2$.

\begin{rem}\label{r-bc}
If $\Omega$ satisfies suitable regularity conditions, then the elements $u \in \dom(A)$
satisfy the conditions $u|_{D}=0$ in the sense of traces and 
$\nu \cdot (\mu \nabla u) =0 $ on $\partial \Omega \setminus D$ in a generalized sense, 
cf.\ \cite[Chapter~1.2]{Cia}, \cite[Chapter~II.2]{GGZ}) or \cite[Chapter~3.3.2]{Lio3}.
Thus, the operator $A$ realizes 
mixed boundary conditions and provides solutions for \eqref{e-prob} in a generalized sense.  
\end{rem}

We call
$D$ the Dirichlet (boundary) part and 
\[
\Upsilon:=\partial \Omega \setminus D
\] the Neumann (boundary)
 part of $\partial \Omega$. 

It is easy to see that the form $\gotl$ satisfies the Beurling--Deny criteria
(see \cite[Corollary~2.17]{EMR}, \cite[Section~2.3]{HKrR} or \cite[Section~4.3]{Ouh5}). 
Hence the semigroup $S$ generated by $-A$ extends consistently to a contraction semigroup
$S^p$ in $L^p(\Omega)$ for all $p \in [1,\infty]$ and $S^p$ is a $C_0$-semigroup
for all $p \in [1,\infty)$.
Let $-A_p$ denote the generator of $S^p$.
If $p \in (2,\infty)$ then $\dom(A_p) = \{ u \in \dom(A) \cap L_p(\Omega) : A u \in L_p(\Omega) \} $
and if $p \in [1,2)$ then $A_p$ is the closure of $A$.
If no confusion is possible, then we write $S = S^p$.

We denote by 
\[
E = \{x = (\tilde {x},x_d): -1 < x_d < 1 
\mbox{ and } \|\tilde {x}\|_{\R^{d-1}}<1\}
\]
the open cylinder in $\R^d$. 
Its lower half is denoted by $E^- = \{ x \in E : x_d < 0 \} $ and 
\[
P = E \cap \{x \in \R^d : x_d=0 \}
  \] 
  is its midplate. 
Furthermore, for all $n\in \mathbb {N}$ and $x \in \R^n$ let $B^n_R(x)$ denote the ball
in $\R^n$ with radius $R$ and centre $x$.
By
$\mathcal H_n$ we denote the $n$-dimensional Hausdorff measure.
We denote the volume of a measurable subset $F \subset \Ri^d$ by $|F|$
and the volume of a measurable subset $F \subset \Ri^{d-1}$ by $\mes_{d-1}(F)$. 

Let $\Omega \subset \R^d$ be open, $M \subset \partial \Omega$ and $\alpha \in (0,1]$.
Then, following \cite[Definition~II.C.1]{KinS} and \cite[Section~1.1]{LaSU}, 
we say that $M$ is of class $(\bf {A_\alpha})$ 
 if 
\[
|B^d_R(x) \setminus  \Omega| \geq \alpha \, |B^d_R(x)|
\]
for all $R \in (0,1]$ and $x \in M$.
It is not hard to see that the boundary of any Lipschitz domain is of class 
$(\bf {A_\alpha})$ for a suitable $\alpha > 0$.
Finally, let us recall the concept of a positive operator, cf.\ \cite[Subsection~1.14.1]{Tri}.

\begin{definition} \label{d-posopera}
A densely defined operator $B$ on a Banach space $X$ is called \emph{positive}, if
there is a $c > 0$ such that 
\[
\|(B+\lambda)^{-1} \|_{\mathcal L(X)} \le \frac {c}{1+\lambda}
\]
for all $\lambda \in [0,\infty)$.
\end{definition}

We next introduce three assumptions  
on the domain $\Omega$ and its boundary $\partial \Omega$. 
Recall that $\Upsilon =\partial \Omega \setminus D$ is the Neumann part of $\partial \Omega$. 

\begin{assu} \label{a-LIps}
For all $x \in \overline \Upsilon$ there is an open neighbourhood
$U_x$ and a bi-Lipschitz map $\phi_x$ from a neighbourhood of 
$\overline{ U_x}$ onto an open subset of $\R^d$, such that 
$\phi_x(U_x)=E$, 
$\phi_x( \Omega \cap  U_x)=E^-$, 
$\phi_x(\partial \Omega \cap  U_x)=P$ and $\phi_x(x)=0$.
\end{assu}

\begin{assu} \label{a-Aalpha}
There is an $\alpha > 0$ such that 
the set $D$ is of class $(\bf {A_\alpha})$.
\end{assu}

\begin{assu} \label{a-RandUps}
Let $\partial \Upsilon$ be the boundary of $\Upsilon$ in $\partial \Omega$. 
For all $x \in \partial \Upsilon$, there are $c_0 \in (0,1)$ and $c_1 >0$ such that
\begin{equation} \label{e-masstheo}
\mes_{d-1} \{\tilde z \in B^{d-1}_R(\tilde {y}) :
\dist(\tilde z,\phi_x(\Upsilon \cap U_x))>c_0 \, R \} 
\ge c_1 \, R^{d-1}
\end{equation}
for all $R \in (0,1]$ and 
$\tilde {y} \in \Ri^{d-1}$ with 
$(\tilde {y},0) \in  \phi_x(\partial \Upsilon \cap U_x)$,
where $U_x$ and $\phi_x$ are as in Assumption \ref{a-LIps}.
\end{assu}

We would like to remark on two consequences of these assumptions.

\begin{rem}
Assumptions \ref{a-LIps} and \ref{a-Aalpha} exclude the presence of cracks in 
$\Omega$ as these cracks would include boundary points which 
satisfy neither the $(\bf {A_\alpha})$-condition nor do they allow for a 
Lipschitz chart satisfying Assumption \ref{a-LIps}.
\end{rem}

\begin{rem}\label{r-RandUps}
Assumption~\ref{a-RandUps} implies  
the `lower bound' in the Ahlfors--David condition (cf.\ \cite[Chapter~II]{JW}),
i.e.\ there is a $\check c_1 >0$ such that 
\[
\mathcal H_{d-1}(D \cap B^d_R(x)) \ge \check c_1 R^{d-1}
\]
for all $x \in D$ and $R \in (0,1]$.
See also \cite[Lemma 5.4]{ERe2}.
\end{rem}

In the sequel we collect results of foregoing papers which will enable us to prove
 parabolic  H\"older estimates.
The following result was shown in \cite[Theorem 1.1]{ERe2}.

\begin{theorem} \label{t-dir513}
Let $\Omega \subset \R^d$ be a bounded domain and $D$ a closed subset 
of the boundary $\partial \Omega$.
Suppose that Assumptions {\rm \ref{a-LIps}}, {\rm \ref{a-Aalpha}} and {\rm \ref{a-RandUps}} are valid.
Then for all $q\in (d,\infty)$ 
there exists a $\kappa>0$ such that 
$\dom(\ca_q) \subset C^\kappa$. 
\end{theorem}

The next theorem concerns properties of $A_p$ and the semigroup 
$S$ generated by $-A$.

\begin{theorem} \label{t-sgrp0}
Let $\Omega \subset \R^d$ be a bounded domain and $D$ a closed subset 
of the boundary $\partial \Omega$.
Suppose that Assumption {\rm \ref{a-LIps}} holds true.
Then one has the following.
\begin{tabel}
\item \label{i-Hinfty} 
For all $p\in(1,\infty)$ and $\lambda_0 >0$ the operator $A_p + \lambda_0$ has a 
bounded $H^\infty$-calculus. 
In particular, $A_p + \lambda_0$ is a positive operator with bounded imaginary powers.
\item \label{i-MR} 
If $p\in(1,\infty)$, then $A_p$ has maximal parabolic $L^s((0,T);L^p)$-regularity.
\item \label{i-upperGauss}
The semigroup $S$ has a kernel $K$ satisfying Gaussian upper bounds.
Stronger: for all $\omega > 0$ there are
$b,c>0$ such that 
\begin{equation}\label{i-kest1}
|K_t(x,y)|
\leq c \, t^{-d/2} \, e^{-b \frac{|x-y|^2}{t}} \, e^{\omega t}
\end{equation}
for all $x,y \in \Omega$ and $t > 0$.
\item \label{i-lrlp}
For all $\omega > 0$ there exists a $c > 0$ such that 
\begin{equation}\label{i-sgest}
\|S_t\|_{\mathcal{L}(L^p,L^r)}
\leq ct^{-\frac{d}{2}(\frac{1}{p}-\frac{1}{r})}e^{\omega t}
\end{equation}
for all $t\in (0,\infty)$ and $p,r \in [1,\infty]$ with $p \leq r$.
\end{tabel}
\end{theorem}

\begin{proof}
Since $S^p$ is a contraction semigroup, Statement~\ref{i-Hinfty} follows
from \cite{Cow1,LeMX,LeMerdy} and Statement~\ref{i-MR} from \cite{Lamb}.

By \cite[Theorem~3.1]{ERe1} there are $b,c,\omega > 0$ such that 
(\ref{i-kest1}) is valid for all $x,y \in \Omega$ and $t > 0$.
Hence there are $c,\omega > 0$ such that (\ref{i-sgest}) is valid for all 
$t \in (0,\infty)$ and $p,r \in [1,\infty]$ with $p \leq r$.
Since $S$ is a contraction semigroup on $L^2$, the bounds on 
$\|S_t\|_{\mathcal{L}(L^2,L^\infty)}$ can be improved by using \cite[Lemma~6.5]{Ouh5}
and there exists a $c > 0$ such that 
$\|S_t\|_{\mathcal{L}(L^2,L^\infty)} \leq c t^{-d/4} (1+t)^{d/4}$
for all $t > 0$.
Duality gives that there exists a $c > 0$ such that 
\[
\|S_t\|_{\mathcal{L}(L^1,L^\infty)} 
\leq c t^{-d/2} (1+t)^{d/2}
\leq c \varepsilon^{-d/2} t^{-d/2} e^{\varepsilon d t/2}
\]
for all $t > 0$ and $\varepsilon \in (0,1]$.
Since $|K_t(x,y)| \leq \|S_t\|_{\mathcal{L}(L^1,L^\infty)}^{1-\varepsilon} |K_t(x,y)|^\varepsilon$
the Gaussian bounds of \cite[Theorem~3.1]{ERe1} give Statement~\ref{i-upperGauss}.
Then Statement~\ref{i-lrlp} follows directly from Statement~\ref{i-upperGauss}. 
\end{proof}

The last two statements in the following theorem 
are corollaries to the results in \cite{ERe2}.

\begin{theorem} \label{t-sgrp}
In addition to the assumptions of Theorem {\rm \ref{t-sgrp0}}, suppose that 
Assumptions~{\rm \ref{a-Aalpha}} and {\rm \ref{a-RandUps}} are valid. 
Then one has the following.
\begin{tabel}
\item \label{i-GHB}
The kernel $K$ of the semigroup $S$ satisfies Gaussian H\"older kernel bounds, 
i.e.\ there are $\kappa_*,b,c,\omega>0$ such that 
\begin{equation}\label{i-kest2}
|K_t(x,y) - K_t(x',y')|
\leq c \, t^{-d/2} \Big( \frac{|x-x'| + |y-y'|}{t^{1/2}} \Big)^{\kappa_*} 
    e^{-b \frac{|x-y|^2}{t}} \, e^{\omega t}
\end{equation}
for all $x,x',y,y' \in \Omega$ and $t > 0$ with 
$|x-x'| + |y-y'| \leq t^{1/2}$.
\item \label{i-Hkappa}
There exists a $c > 0$ such that 
\[
\|S_t\|_{\mathcal{L}(L^p,C^{\kappa})}
\leq c \, t^{-\frac{d}{2p}-\frac{\kappa}{2}}e^{\omega t}
\]
for all $p\in[1,\infty]$, $\kappa \in (0,\kappa_*]$
and $t\in(0,\infty)$, where $\kappa_*$ and $\omega$ are as in {\rm \ref{i-GHB}}.
\item  \label{t-sgrp-7}
Let $\kappa_*$ be as in {\rm \ref{i-GHB}}.
Then 
\[
\dom\bigl (A_p^\theta\bigr ) \hookrightarrow C^\kappa.
\]
for all $p \in [1,\infty)$, $\kappa\in (0,\kappa_*]$ 
and $\theta \in (\frac{d}{2p}+\frac{\kappa}{2},\infty)$.
\end{tabel}
\end{theorem} 
\begin{proof}
Statement~\ref{i-GHB} was shown in \cite[Theorem~7.5]{ERe2}.
We next show Statement~\ref{i-Hkappa}. 
Let $u \in L^p$ and let $x,x' \in \Omega$ with $0<|x-x'| \leq 1$. 
Let $t>0$.
We consider two cases. 

\noindent
{\em Case 1.} Suppose that $|x-x'| \leq t^{1/2}$. \\
Then \eqref{i-kest2} implies that 
\begin{align*}
|(S_t u )(x) - (S_t u)(x')| & \leq \int_\Omega |K_t(x,y)-K_t(x',y)| \, |u(y)|\, dy \\
& \leq c \Big(\frac{|x - x'|}{t^{1/2}}\Big)^\kappa 
    \int_\Omega t^{-d/2}e^{-b\frac{|x-y|^2}{t}}e^{\omega t} |u(y)| \, dy \\
& \leq c \Big(\frac{|x - x'|}{t^{1/2}}\Big)^\kappa t^{-\frac{d}{2p}} 
      e^{\omega t} \|u\|_{L^p},
\end{align*}
where the last step follows from the H\"older inequality.

\noindent
{\em Case 2.} Suppose that $|x-x'| \geq t^{1/2}$. \\
Then trivially,
\[
|(S_t u )(x) - (S_t u)(x')| \leq 2 \|S_t u\|_{L^\infty} 
\leq 2 ct^{-\frac{d}{2p}}e^{\omega t} \|u\|_{L^p} 
\leq 2 c \Big(\frac{|x - x'|}{t^{1/2}}\Big)^\kappa t^{-\frac{d}{2p}}e^{\omega t} \|u\|_{L^p},
\]
where in the second step \eqref{i-sgest} was used with $r=\infty$. 

A combination of both cases implies Statement~\ref{i-Hkappa}. 

Statement~\ref{t-sgrp-7} follows from Statement~\ref{i-Hkappa} and 
the integral representation 
\[
B^{-\theta}
=\frac {1}{\Gamma(\theta)} \int_0^\infty t^{\theta-1} e^{-tB} \, d t, 
\]
(see \cite[(2.6.9)]{Paz}), applied to $B=A_p+\omega+1$. We obtain  
$\dom((A_p + \omega + 1)^\theta) \hookrightarrow C^\kappa$.
But $\dom((A_p + \omega + 1)^\theta) = \dom\bigl (A_p^\theta\bigr )$ with equivalent norms.
\end{proof}

\begin{rem} \label{r-inter}
By Theorem \ref{t-sgrp0}\ref{i-Hinfty} the operator $A_p+1$ admits bounded imaginary powers.
Hence 
\[
\dom(A_p^\theta)=[L^p,\dom(A_p+1)]_\theta
\] 
by \cite[Theorem~1.15.3]{Tri}, if in addition $\theta < 1$.
So $[L^p,\dom(A_p)]_\theta \hookrightarrow C^\kappa$
by Theorem~\ref{t-sgrp}\ref{t-sgrp-7}.
\end{rem}


\section{H\"older regularity for parabolic problems in $L^p$} \label{s-parabolicLp}

We interpret parabolic problems of the form \eqref{e-prob} as the abstract 
Cauchy problems associated to $A$.
Our first theorem is the following.

\begin{theorem} \label{t-parahoeldlp}
Adopt the notation and assumptions as in Theorem~{\rm \ref{t-sgrp}}.
Let $\kappa_*$ be as in Theorem~{\rm \ref{t-sgrp}\ref{i-GHB}}.
Let $T > 0$ and write $J = (0,T)$.
Let $p \in (\frac {d}{2},\infty)$, $\kappa \in (0,\kappa_*]$, $\theta \in (0,1)$
and $s \in (1,\infty)$.
Suppose that 
\[
\frac{d}{2p}+\frac{\kappa}{2} < \theta < 1-\frac{1}{s}
{}.  \]
Then there are $c > 0$ and $\beta \in (0,1)$ such that the following is valid.
Let $f \in L^s(J;L^p)$ and 
$ u_0 \in X_{s,p} := (L^p,\dom(A_p))_{1-\frac{1}{s},s}$.
Then any solution $u$ of the equation 
\begin{equation} \label{e-sol1}
u'+A u=f \in L^s(J;L^p),\quad u(0)=u_0,
\end{equation}
 belongs to $C^\beta(J;C^\kappa)$ and 
\[
\|u\|_{C^\beta(J;C^\kappa)} \le c (\|f\|_{L^s(J;L^p)} + \Vert u_0 \Vert_{X_{s,p}}),
\]
where $\beta=1-\frac {1}{s} -\theta$.
\end{theorem}

Note that 
\[
C^\beta(J;C^\kappa) \subset C^{\min(\beta,\kappa)}(J \times \Omega)
{}.  \]

In preparation for the proof of this theorem, we first recall the notion 
of maximal parabolic regularity.

\begin{definition} \label{d-maxreg}
Let $s \in (1,\infty)$ and let $X$ be a Banach space.
Assume that $B$ is a densely defined closed
operator in $X$. 
Let $T > 0$ and set $J = (0,T)$.
We say that $B$ satisfies \emph {maximal parabolic $L^s(J; X)$ regularity}, 
if there is an isomorphism which maps every
$f\in L^s(J; X)$ to the unique function $u \in W^{1,s}(J; X) \cap
L^s(J; \dom(B))$ satisfying
\[
u' + Bu = f,\quad \quad u(0) = 0.
\]
\end{definition}

\begin{rem} \label{r-ebed}
We recall the following results associated to Definition \ref{d-maxreg}. 
\begin{itemize}
\item 
The property of maximal parabolic $L^s(J;X)$ regularity of
an operator $B$ is independent of the summability index 
$s \in (1, \infty)$ and
the choice of $T$ for the interval $J$, cf.\ \cite{Dor}.
We will say for short that $B$ admits maximal parabolic regularity on $X$.
\item 
If an operator satisfies maximal parabolic regularity on a Banach space
$X$, then its negative generates an analytic semigroup on $X$, cf.\ \cite{Dor}.
In particular, a suitable left half-plane belongs to its resolvent set.
\item
Let $X$ be a Banach space and let $s \in (1,\infty)$ and $T > 0$.
Set $J = (0,T)$.
Let $B$ be an operator in $X$ which admits maximal parabolic regularity.
Then there exists a $c > 0$ such that for all $f \in L^s(J;X)$ and 
$u_0 \in (X,\dom(B))_{1-\frac{1}{s},s}$ there exists a unique
$u \in W^{1,s}(J; X) \cap L^s(J; \dom(B))$ such that 
\[
u' + Bu = f,\quad \quad u(0) = u_0.
\]
Moreover, 
\[
\|u\|_{W^{1,s}(J; X) \cap L^s(J; \dom(B))}
\leq c ( \|f\|_{L^s(J;X)} + \|u_0\|_{(X,\dom(B))_{1-\frac{1}{s},s}}),
\]
cf.\ \cite[Proposition~2.1 (i)$\Rightarrow$(iii)]{Ama4}.
\end{itemize}
\end{rem}

The space of maximal parabolic regularity allows for the following embedding results.

\begin{lemma} \label{l-inthoelpar}
Let $X,Y$ be Banach spaces and assume that $Y$ is continuously embedded into~$X$.
Let $T > 0$ and set $J = (0,T)$. 
\begin{tabel}
\item \label{l-contemb}
If $s\in (1,\infty)$, then
\[
W^{1,s}(J;X) \cap L^s(J;Y) \hookrightarrow C(\overline J;(X,Y)_{1-\frac {1}{s},s}).
\]  
\item \label{l-hoeldemb}
If $s \in (1,\infty)$ and $\theta \in (0,1-\frac {1}{s})$, then 
\[
W^{1,s}(J;X) \cap L^s(J;Y) \hookrightarrow 
C^\beta(J;(X,Y)_{\theta,1}),
\] 
where $\beta=1-\frac {1}{s}-\theta$.
\end{tabel}
\end{lemma}
\begin{proof}
The first part of the lemma is proved in \cite[Chapter III, Theorem~4.10.2]{Ama2}.

In order to prove \ref{l-hoeldemb}, we first note that 
\begin{eqnarray*}
\|w(t_1)-w(t_2)\|_X 
& = & \|\int_{t_1}^{t_2}w'(t) \,dt\|_X 
\le \int_{t_1}^{t_2}\|w'(t)\|_X \,dt  \\
& \le & \Bigl (\int_{t_1}^{t_2} \|w'(t)\|_X^s \, dt \Bigr )^{1/s} |t_1-t_2|^{1-1/s}  \\
& \le & \Bigl (\int_J \|w'(t)\|_X^s \, dt \Bigr )^{1/s} |t_1-t_2|^{1-1/s}  \\
& \leq & \|w\|_{W^{1,s}(J;X)} |t_1-t_2|^{1-1/s} 
\end{eqnarray*}
for all $w \in W^{1,s}(J;X)$ and $t_1, t_2 \in J$ with $t_1 < t_2$.
Moreover, since $0<\theta < 1-\frac {1}{s}$, the reiteration theorem 
\cite[Theorem~1.10.2]{Tri} gives
\[
(X,Y)_{\theta,1}=(X,(X,Y)_{1-\frac {1}{s},s})_{\lambda,1}
,\]
where $\lambda:= \frac {\theta}{1-\frac {1}{s}} < 1$.
Then $\beta=(1-\lambda)(1-\frac {1}{s})$ and
\begin{eqnarray*}
\frac {\|w(t_1)-w(t_2)\|_{(X,Y)_{\theta,1}}}{|t_1-t_2|^\beta} 
& \le & \frac {\|w(t_1)-w(t_2)\|_{X}^{1-\lambda}}{|t_1-t_2|^\beta} 
   \Big(\sum_{j=1}^2\|w(t_j)\|_{(X,Y)_{1-\frac {1}{s},s}} \Big)^\lambda   \\
& \le & \Bigl (\frac {\|w(t_1)-w(t_2)\|_{X}}{|t_1-t_2|^{1-\frac {1}{s}}} \Bigr )^{1-\lambda}
   2 \bigl (\sup_{t \in J}\|w(t)\|_{(X,Y)_{1-\frac {1}{s},s}}\bigr )^\lambda \\
& \le & 2 c^\lambda \,\|w\|_{W^{1,s}(J;X) \cap L^s(J;Y)},
\end{eqnarray*}
where $c$ is the norm of the inclusion in Statement~\ref{l-contemb}.
\end{proof}

\begin{rem} \label{r-amann}
As \ref{l-contemb}, also \ref{l-hoeldemb} in Lemma \ref{l-inthoelpar} is known, cf.\ \cite[Theorem~3]{Ama3}, 
but our  proof is elementary.
\end{rem}

\begin{proof}[{\bf Proof of Theorem \ref{t-parahoeldlp}}]
By Theorem~\ref{t-sgrp0}\ref{i-MR} the operator $A_p$ has maximal parabolic regularity in $L^p$.
Therefore the third point in Remark~\ref{r-ebed} gives that 
the solution $u$ of \eqref{e-sol1} belongs to the space $W^{1,s}(J;L^p) \cap 
L^s(J;\dom(A_p))$ with the estimate
\[
\|u\|_{W^{1,s}(J,L^p) \cap L^s(J;\dom(A_p))} 
\le c (\|f\|_{L^s(J;L^p)} + \Vert u_0 \Vert_{X_{s,p}})
\]
for a suitable $c > 0$.
Putting $X:=L^p$ and $Y:=\dom(A_p)$, Lemma \ref{l-inthoelpar}\ref{l-hoeldemb}
gives $u \in C^\beta(J;(L^p,\dom(A_p))_{\theta,1})$
including the estimate
\begin{equation} \label{e-absch01}
\|u\|_{C^\beta(J;(L^p,\dom(A_p))_{\theta,1})} 
\le c_1\, \|u\|_{W^{1,s}(J,L^p) \cap L^s(J;\dom(A_p))} 
\le c\, c_1 ( \|f\|_{L^r(J;L^p)} + \Vert u_0 \Vert_{X_{s,p}})
\end{equation}
for a suitable $c_1 > 0$.
Since $A_p + 1$ is a positive operator on $L^p$, we have the continuous embedding
\[
(L^p,\dom(A_p))_{\theta,1}
= (L^p,\dom(A_p + 1))_{\theta,1} 
\hookrightarrow \dom\bigl ((A_p + 1)^\theta\bigr )
= \dom(A_p^\theta)
\]
by \cite[Theorem~1.15.2(d)]{Tri}.
This, combined with \eqref{e-absch01} and Theorem~\ref{t-sgrp}\ref{t-sgrp-7}, gives the claim.
\end{proof}


\section{H\"older regularity for parabolic problems in $W^{-1,q}_D$} \label{SHoelpara4}

The treatment of parabolic equations in $L^p$ spaces is quite common; let us 
therefore start this section with some motivation for the consideration
of parabolic equations in $W^{-1,q}_D$. 
If the right hand side  
of the equation (considered at any time point) has a Lebesgue density in the domain 
and if the boundary condition is either homogeneous or purely
Dirichlet, then, e.g.\ $L^p$ spaces are adequate.
Naturally, spaces of type $W^{-1,q}$ come into play when the right-hand side is 
given by a distributional object, as e.g.\ surface charge densities
or thermal sources, concentrated on a $(d-1)$-dimensional surface.
These spaces may also be adequate for studying inhomogeneous Neumann boundary conditions, see 
\cite[Chapter~3.2]{Lio3}, for example, if the right-hand-side in the first equation in 
\eqref{e-prob} is given by $f \in L^s(J;L^{m_0})$  and $0$ on the right hand side of the third 
equation in \eqref{e-prob} is replaced by a function $g \in L^s(J;L^{m_1}(\Upsilon))$ 
with suitable $m_0(d,q), m_1(d,q)\in [1,\infty)$, one can define $F \in L^s(J;W_D^{-1,q})$
 by 
\[
F(t)(\phi) = \int_\Omega f(t) \phi  + \int_{\Upsilon} g(t) \, \phi|_\Upsilon ,\quad
\text{for all }\phi \in W_D^{1,q'},  
\]  and choose $F$ as the right-hand-side in 
the abstract Cauchy problem (see \eqref{e-sol2} below). 
Note also that in general, one cannot replace the condition $f \in L^s(J;W_D^{-1,q})$,
where $q \in (d,\infty)$, by $f \in L^s(J;W_D^{-1,2})$,
because this would not necessarily yield the regularity 
which is needed in particular for the treatment of non-linear problems. 
The aim of this section is to show that for all $q \in (d,\infty)$
the solutions of the parabolic 
problem in $W^{-1,q}_D$ are  
H\"older continuous in space and time.

In order to state the main result of this section, we must introduce 
additional assumptions on $D$ and~$\mu$. 

\begin{assu} \label{a-Ahlf}
Either $D=\emptyset$ or $D$ satisfies the \emph{Ahlfors--David condition}:
There are constants $c_0, c_1 > 0$ and $r_{AD} > 0$, such that 
\begin{equation} \label{e-ahlf}
  c_0 R^{d-1} \le \mathcal H_{d-1} (D \cap B^d_R(x) ) \le c_1 R^{d-1}
\end{equation}
for all $x \in D$ and  $R \in (0,r_{AD}]$.
\end{assu}

\begin{rem} \label{r-surfmeas}
Assumption \ref{a-Ahlf} means the following.
\begin{tabel} 
\item 
The set $D$ is a $(d-1)$-set in the sense 
of Jonsson/Wallin \cite[Chapter~II]{JW}.
\item \label{r-surfmeas:ii} 
On the set $\partial \Omega \cap \bigl(
\bigcup_{x \in  \partial \Upsilon} U_x
\bigr)$, the measure $\mathcal H_{d-1}$ equals the surface measure
$\sigma$ which can be constructed via the bi-Lipschitz charts
$\phi_x$ given in Assumption~\ref{a-LIps}, 
cf.\ \cite[Subsection~3.3.4~C]{EvG} or \cite[Section~3]{HaR2}.
In particular, \eqref{e-ahlf} implies that $\sigma \bigl( D \cap
\bigl( \bigcup _{x \in \partial \Upsilon} U_x
\bigr) \bigr) > 0$, if $\partial \Omega \neq D \neq \emptyset$.
 \end{tabel}
\end{rem}

\begin{assu} \label{assu-coeffi}
$\dom\bigl ((A+1)\bigr )^\frac {1}{2} = W^{1,2}_D$.
\end{assu}

\begin{rem} \label{r-valid}
Assumption~\ref{assu-coeffi} is 
\emph{not} known for arbitrary non-symmetric coefficient functions under
our general assumptions on the geometry of $\Omega$ and $D$.
But many
special cases are available: 
\begin{tabel} 
  \item If Assumption~\ref{assu-coeffi} is satisfied for some coefficient function $\mu$, then
it is also true for the adjoint coefficient function,
cf.\ \cite[Theorems~1 and 2]{Kat3}.
  \item Assumption~\ref{assu-coeffi} is always fulfilled if the coefficient
function $\mu$ takes its values in the set of real \emph{symmetric}
$d\times d$-matrices.
  \item For results on non-symmetric coefficient functions, see \cite{AKM2}.
By a recent result in \cite[Theorem~4.1]{EHT}, Assumption~\ref{assu-coeffi} is valid in 
our geometric setting, if the domain $\Omega$ itself is a $d$-set, cf.\ \cite[Chapter~II]{JW}.
\end{tabel}
\end{rem}

Let us now state the second main result of this paper.

\begin{theorem} \label{t-parahoeld}
Adopt the notation and assumptions as in Theorem {\rm \ref{t-sgrp}} and, in addition,
adopt Assumptions {\rm \ref{a-Ahlf}} and {\rm \ref{assu-coeffi}}. 
Let $\kappa_*$ be as in Theorem~{\rm \ref{t-sgrp}\ref{i-GHB}}.
Moreover, let $q \in (d,\infty)$, $\kappa \in (0, \kappa_*]$, $\theta \in (0,1)$ 
and $s\in (1,\infty)$ be such that 
$\frac{d}{2q}+\frac{\kappa}{2} +\frac{1}{2} < \theta < 1-\frac{1}{s}$.
Then there exists a $c > 0$ such that the following is valid.
Let $f \in L^s(J;W^{-1,q}_D)$ and $u_0 \in X_{s,-1,q}:= (W^{-1,q}_D,\dom(\ca_q))_{1-\frac{1}{s},s}$.
Then any solution $u$ of the equation 
\begin{equation} \label{e-sol2}
u'+\ca_q u=f 
,\quad 
u(0)=u_0,
\end{equation}
belongs to $C^\beta(J;C^\kappa)$ and
\[
\|u\|_{C^\beta(J;C^\kappa)} \le c ( \|f\|_{L^s(J;W^{-1,q}_D)} + \| u_0 \Vert_{X_{s,-1,q}}),
\]
where $\beta = 1-\frac {1}{s} -\theta$.
\end{theorem}

For the proof of this theorem, we need some additional results from  
\cite[Section~11]{ABHR}.

\begin{theorem} \label{t-mainsect}
Adopt~Assumptions~{\rm \ref{a-LIps}}, {\rm \ref{a-Ahlf}} and {\rm \ref{assu-coeffi}}.
Let $q \in [2,\infty)$.
Then one has the following.
\begin{tabel}
\item \label{t-carryover:i} 
$\ca_q+1$ is a positive operator in $W^{-1,q}_D$.
\item \label{t-carryover:ii}
$(\ca_q+1)^{-1/2}$ provides a topological 
isomorphism between $W^{-1,q}_D$ and $L^q$.
\item \label{t-carryover:iii} 
$\ca_q$ admits maximal parabolic regularity in $W^{-1,q}_D$.
 \end{tabel}
\end{theorem}

We exploit Theorem \ref{t-mainsect} for the proofs of the following lemmas.

\begin{lemma}\label{l-domApAq}
Adopt the notation and assumptions as in Theorem {\rm \ref{t-sgrp}} and, in addition,
adopt Assumptions {\rm \ref{a-Ahlf}} and {\rm \ref{assu-coeffi}}. 
If $\theta \in (\frac{1}{2},1)$, $\varsigma \in [1,\infty]$ and $q \in (2,\infty)$ then
\[
(W^{-1,q}_D,\dom(\ca_q))_{\theta,\varsigma}
=(L^q,\dom(A_q))_{\theta - \frac {1}{2},\varsigma}.
\]
\end{lemma}
\begin{proof}
Theorem \ref{t-mainsect}\ref{t-carryover:ii} gives $\dom((\ca_q + 1)^{1/2})=L^q$, which implies
that $\dom(\ca_q + 1)=\dom((A_q+1)^{1/2})$.
By Theorems~\ref{t-sgrp}\ref{i-MR} and \ref{t-mainsect}\ref{t-carryover:i}
both operators, $\ca_q +1$ and $A_q+1$,
are positive in $W^{-1,q}_D$ and $L^q$, respectively.
By \cite[Subsection 1.10.1 and Theorem~1.15.2(d)]{Tri} the space
$\dom(\ca_q + 1)^{1/2}$ 
belongs to the class $J(\frac {1}{2}) \cap K(\frac {1}{2})$ between the spaces 
$W^{-1,q}_D$ and $\dom(\ca_q )$ and the space 
$\dom(A_q+1)^{1/2}$ belongs to the class 
$J(\frac {1}{2}) \cap K(\frac {1}{2})$ between the spaces $L^q$ and $\dom(A_q+1)$. 
Therefore the reiteration theorem for real interpolation
\cite[Theorem~1.10.2]{Tri} gives
\begin{align*}
(W^{-1,q}_D,\dom(\ca_q))_{\theta,\varsigma} 
&  =(\dom((\ca_q+1)^{1/2}),\dom(\ca_q+1))_{2\theta-1,\varsigma}\\
& =(L^q,\dom((A_q+1)^{1/2}))_{2\theta-1,\varsigma}  \\
& =(L^q,\dom(A_q+1))_{\theta - \frac {1}{2},\varsigma}
\end{align*}
as requested.
\end{proof}

\begin{lemma}\label{l-fracpower2}
Adopt the notation and assumptions as in Theorem {\rm \ref{t-sgrp}} and, in addition,
adopt Assumptions {\rm \ref{a-Ahlf}} and {\rm \ref{assu-coeffi}}. 
Let $\kappa_*$ be as in Theorem~{\rm \ref{t-sgrp}\ref{i-GHB}},
$\kappa \in (0,\kappa_*]$
and $\theta \in (0,1)$ with $\theta > \frac{d}{2q} + \frac{\kappa}{2} + \frac{1}{2}$.
Then
\[
(W^{-1,q}_D,\dom(\ca_q))_{\theta,1} \hookrightarrow C^\kappa.
\]
\end{lemma}
\begin{proof}
If follows from Lemma~\ref{l-domApAq} and \cite[Theorem~1.15.2(d)]{Tri}
that 
\[
(W^{-1,q}_D,\dom(\ca_q))_{\theta,1}
= (L^q,\dom(A_q + 1))_{\theta - \frac {1}{2},1} 
\subset \dom((A_q + 1)^{\theta -\frac {1}{2}}).
\]
Now an application of Theorem\ref{t-sgrp}\ref{t-sgrp-7}
gives the claim.
\end{proof}

\begin{proof}[{\bf Proof of Theorem \ref{t-parahoeld}}]
By Theorem~\ref{t-mainsect}\ref{t-carryover:iii}, the solution satisfies 
$u \in W^{1,s}(J,W^{-1,q}_D) \cap L^s(J;\dom(\ca_q))$
and there is a suitable $c > 0$ such that 
\[
\|u\|_{W^{1,s}(J,W^{-1,q}_D) \cap L^s(J;\dom(\ca_q))} 
\le c (\|f\|_{L^s(J;W^{-1,q}_D)} + \| u_0 \|_{X_{s,-1,q}}).
\]
Next Lemma~\ref{l-inthoelpar} gives
\[
W^{1,s}(J;W^{-1,q}_D) \cap L^s(J;\dom(\ca_q)) \hookrightarrow 
C^\beta(J;(W^{-1,q}_D,\dom(\ca_q))_{\theta,1})
{}.  \] 
Then the theorem is a consequence of Lemma~\ref{l-fracpower2}.
\end{proof}


\begin{thebibliography}{ABHR}

\bibitem[Ama1]{Ama2}
{\sc Amann, H.}, {\em Linear and quasilinear parabolic problems}.
\newblock Monographs in Mathematics 89. Birkh{\"a}user, Boston, 1995.

\bibitem[Ama2]{Ama3}
\leavevmode\vrule height 2pt depth -1.6pt width 23pt, Linear parabolic problems
  involving measures.
\newblock {\em RACSAM. Rev. R. Acad. Cienc. Exactas Fís. Nat. Ser. A Mat.} {\bf
  95} (2001),  85--119.

\bibitem[Ama3]{Ama4}
\leavevmode\vrule height 2pt depth -1.6pt width 23pt, Nonautonomous parabolic
  equations involving measures.
\newblock {\em J. Math. Sci.} {\bf 130} (2005),  4780--4802.

\bibitem[ABHR]{ABHR}
{\sc Auscher, P., Badr, N., Haller-Dintelmann, R. {\rm and} Rehberg, J.}, The
  square root problem for second order, divergence form operators with mixed
  boundary conditions on $L^p$.
\newblock {\em J. Evol. Eq.} {\bf 15} (2015),  165--208.

\bibitem[AKM]{AKM2}
{\sc Axelsson, A., Keith, S. {\rm and} M$^{\rm c}$Intosh, A.}, The Kato square
  root problem for mixed boundary value problems.
\newblock {\em J. London Math. Soc. (2)} {\bf 74} (2006),  113--130.

\bibitem[Cia]{Cia}
{\sc Ciarlet, P.~G.}, {\em The finite element method for elliptic problems}.
\newblock Studies in Mathematics and its Applications 4. North-Holland,
  Amsterdam, 1978.

\bibitem[Cow]{Cow1}
{\sc Cowling, M.~G.}, Harmonic analysis on semigroups.
\newblock {\em Ann.\ Math.} {\bf 117} (1983),  267--283.

\bibitem[Dor]{Dor}
{\sc Dore, G.}, Maximal regularity in $L^p$ spaces for an abstract Cauchy
  problem.
\newblock {\em Adv. Differential Equations} {\bf 5} (2000),  293--322.

\bibitem[EG]{EvG}
{\sc Evans, L.~C. {\rm and} Gariepy, R.~F.}, {\em Measure theory and fine
  properties of functions}.
\newblock Studies in advanced mathematics. CRC Press, Boca Raton, 1992.

\bibitem[EHT]{EHT}
{\sc Egert, M., Haller-Dintelmann, R. {\rm and} Tolksdorf, P.}, The Kato square
  root problem for mixed boundary conditions.
\newblock {\em J. Funct. Anal.} {\bf 267} (2014),  1419--1461.

\bibitem[EMR]{EMR}
{\sc Elst, A. F.~M. ter, Meyries, M. {\rm and} Rehberg, J.}, Parabolic
  equations with dynamical boundary conditions and source terms on interfaces.
\newblock {\em Ann. Mat. Pura Appl.} {\bf 193} (2014),  1295--1318.

\bibitem[ER1]{ERe1}
{\sc Elst, A. F.~M. ter {\rm and} Rehberg, J.}, $L^\infty$-estimates for
  divergence operators on bad domains.
\newblock {\em Anal. and Appl.} {\bf 10} (2012),  207--214.

\bibitem[ER2]{ERe2}
\leavevmode\vrule height 2pt depth -1.6pt width 23pt, H{\"o}lder estimates for
  second-order operators on domains with rough boundary.
\newblock {\em Adv. Diff. Equ.} {\bf 20} (2015),  299--360.

\bibitem[GGZ]{GGZ}
{\sc Gajewski, H., Gr{\"o}ger, K. {\rm and} Zacharias, K.}, {\em Nichtlineare
  Operatorglei\-chungen und Operatordifferentialgleichungen}.
\newblock Mathematische Lehrb{\"u}cher und Monographien, II. Abteilung
  Mathematische Monographien 38. Akademie-Verlag, Berlin, 1974.

\bibitem[Gri1]{Grie2}
{\sc Griepentrog, J.~A.}, Maximal regularity for nonsmooth parabolic problems
  in Sobolev-Morrey spaces.
\newblock {\em Adv. Differential Equations} {\bf 12} (2007),  1031--1078.

\bibitem[Gri2]{Grie4}
\leavevmode\vrule height 2pt depth -1.6pt width 23pt, Sobolev--Morrey spaces
  associated with evolution equations.
\newblock {\em Adv. Differential Equations} {\bf 12} (2007),  781--840.

\bibitem[GKR]{GKR}
{\sc Griepentrog, J.~A., Kaiser, H.-C. {\rm and} Rehberg, J.}, Heat kernel and
  resolvent properties for second order elliptic differential operators with
  general boundary conditions on $L^p$.
\newblock {\em Adv. Math. Sci. Appl.} {\bf 11} (2001),  87--112.

\bibitem[Gri]{Gris}
{\sc Grisvard, P.}, {\em Elliptic problems in nonsmooth domains}.
\newblock Monographs and Studies in Mathematics 24. Pitman, Boston etc., 1985.

\bibitem[Gr{\"o}]{Groe}
{\sc Gr{\"o}ger, K.}, A $W^{1,p}$-estimate for solutions to mixed boundary
  value problems for second order elliptic differential equations.
\newblock {\em Math. Anal.} {\bf 283} (1989),  679--687.

\bibitem[HaR]{HaR2}
{\sc Haller-Dintelmann, R. {\rm and} Rehberg, J.}, Coercivity for elliptic
  operators and positivity of solutions on Lipschitz domains.
\newblock {\em Arch. Math.} {\bf 95} (2010),  457--468.

\bibitem[HiR]{HiebR}
{\sc Hieber, M. {\rm and} Rehberg, J.}, Quasilinear parabolic systems with
  mixed boundary conditions on nonsmooth domains.
\newblock {\em SIAM J. Math. Anal.} {\bf 40} (2008),  292--305.

\bibitem[HKR]{HKrR}
{\sc H{\"o}mberg, D., Krumbiegel, K. {\rm and} Rehberg, J.}, Optimal control of
  a parabolic equation with dynamic boundary condition.
\newblock {\em Appl. Math. Optim.} {\bf 67} (2013),  3--31.

\bibitem[JW]{JW}
{\sc Jonsson, A. {\rm and} Wallin, H.}, Function spaces on subsets of ${\bf
  R}^n$.
\newblock {\em Math. Rep.} {\bf 2}, No.\ 1 (1984).

\bibitem[Kat]{Kat3}
{\sc Kato, T.}, Fractional powers of dissipative operators, II.
\newblock {\em J. Math.\ Soc.\ Japan} {\bf 14} (1962),  242--248.

\bibitem[KS]{KinS}
{\sc Kinderlehrer, D. {\rm and} Stampacchia, G.}, {\em An introduction to
  variational inequalities and their applications}.
\newblock Pure and Applied Mathematics 88. Academic Press, New York, 1980.

\bibitem[LSU]{LaSU}
{\sc Ladyzhenskaya, O.~A., Sollonnikov, V.~A. {\rm and} Ural'tseva, N.~N.},
  {\em Linear and quasilinear elliptic equations of parabolic type}.
\newblock Translations of Mathematical Monographs 23. Amer. Math. Soc.,
  Providence, RI, 1968.

\bibitem[Lam]{Lamb}
{\sc Lamberon, D.}, \'Equations d'\'evolution lin\'eaires associ\'ees \`a des
  semi-groupes de contractions dans les espaces $L^p$.
\newblock {\em J. Funct. Anal.} {\bf 72} (1987),  252--262.

\bibitem[LeM]{LeMerdy}
{\sc {Le Merdy}, C.}, $H^\infty$-functional calculus and applications to
  maximal regularity.
\newblock In {\em Semigroups d\'op\'erateurs et calcul fonctionnel}, Publ.
  Math. UFR Sci. Tech. Besan\c{c}on 16,  41--77. Univ. Franche-Comt\'e,
  Besan\c{c}on, 1998.

\bibitem[LX]{LeMX}
{\sc {Le Merdy}, C. {\rm and} Xu, Q.}, Maximal theorems and square functions
  for analytic operators on $L^p$-spaces.
\newblock {\em J. Lond. Math. Soc.} {\bf 86} (2012),  343--365.

\bibitem[Lie]{Liebe3}
 {\sc Lieberman, G.~M.}, {\em Second order parabolic differential equations}.
\newblock World Scientific Publishing Co., Inc., River Edge, NJ, 1996.

\bibitem[Lio]{Lio3}
{\sc Lions, J.~L.}, {\em Contr\^ole optimal de syst\`emes gouvern\'es par des
  \'equations aux d\'eriv\'ees partielles}.
\newblock Dunod, Paris; Gauthier-Villars, Paris, 1968.

\bibitem[Mos1]{Mos2}
{\sc Moser, J.}, On the regularity problem for elliptic and parabolic
  differential equations.
\newblock In {\em Partial differential equations and continuum mechanics},
  159--169. Univ. of Wisconsin Press, Madison, Wis, 1961.

\bibitem[Mos2]{Mos}
\leavevmode\vrule height 2pt depth -1.6pt width 23pt, A Harnack inequality for
  parabolic differential equations.
\newblock {\em Commun.\ Pure Appl.\ Math.} {\bf 17} (1964),  101--134.
\newblock Correction to: ``A Harnack inequality for parabolic differential
  equations'', Comm.\ Pure Appl.\ Math.\ {\bf 20} (1967), 231--236.

\bibitem[Nas]{Nash}
{\sc Nash, J.}, Continuity of solutions of parabolic and elliptic equations.
\newblock {\em Amer.\ J. Math.} {\bf 80} (1958),  931--954.

\bibitem[Ouh]{Ouh5}
{\sc Ouhabaz, E.-M.}, {\em Analysis of heat equations on domains}, vol.\ 31 of
  London Mathematical Society Monographs Series.
\newblock Princeton University Press, Princeton, NJ, 2005.

\bibitem[Paz]{Paz}
{\sc Pazy, A.}, {\em Semigroups of linear operators and applications to partial
  differential equations}.
\newblock Applied mathematical sciences 44. Springer-Verlag, New York etc.,
  1983.

\bibitem[Sta]{Stam2}
{\sc Stampacchia, G.}, Le Probl\`eme de Dirichlet pour les \'equations
  elliptiques du second ordre \'a coefficients discontinus.
\newblock {\em Ann. Inst. Fourier, Grenoble} {\bf 15} (1965),  189--258.

\bibitem[Tri]{Tri}
{\sc Triebel, H.}, {\em Interpolation theory, function spaces, differential
  operators}.
\newblock North-Holland, Amsterdam, 1978.

\end{thebibliography}

\small
\noindent
{\sc K. Disser,
Weierstrass Institute for Applied Analysis and Stochastics,
Mohrenstr.~39, 
10117 Berlin, 
Germany}  \\
{\em E-mail address}\/: {\bf karoline.disser@wias-berlin.de}

\mbox{}

\noindent
{\sc A.F.M. ter Elst,
Department of Mathematics,
University of Auckland,
Private bag 92019,
Auckland 1142,
New Zealand}  \\
{\em E-mail address}\/: {\bf terelst@math.auckland.ac.nz}

\mbox{}

\noindent
{\sc J. Rehberg,
Weierstrass Institute for Applied Analysis and Stochastics,
Mohrenstr.~39, 
10117 Berlin, 
Germany}  \\
{\em E-mail address}\/: {\bf rehberg@wias-berlin.de}

\end{document}